\documentclass[12pt]{amsart}
\usepackage{amsmath,amsthm,amsfonts,amssymb,times,latexsym,mathabx,mathrsfs,bbm,url}

\theoremstyle{remark}

\theoremstyle{plain}

\newtheorem{lem}{Lemma}[section]
\newtheorem{thm}{Theorem}

\numberwithin{equation}{section}

\makeatletter
\renewcommand{\pmod}[1]{\allowbreak\mkern7mu({\operator@font mod}\,\,#1)}
\makeatother

\newcommand{\ZZ}{{\mathbb Z}}

\newcommand{\cA}{\ensuremath{\mathcal{A}}}
\newcommand{\cB}{\ensuremath{\mathcal{B}}}

\newcommand{\cE}{\ensuremath{\mathcal{E}}}
\newcommand{\cF}{\ensuremath{\mathcal{F}}}

\newcommand{\cJ}{\ensuremath{\mathcal{J}}}
\newcommand{\cK}{\ensuremath{\mathcal{K}}}

\newcommand{\cN}{\ensuremath{\mathcal{N}}}

\newcommand{\cR}{\ensuremath{\mathcal{R}}}

\newcommand{\cV}{\ensuremath{\mathcal{V}}}

\newcommand{\PR}{\mathbb{P}}
\newcommand{\E}{\mathbb{E}}

\renewcommand{\ssum}[1]{\sum_{\substack{#1}}}    
\newcommand{\fl}[1]{{\ensuremath{\left\lfloor {#1} \right\rfloor}}}  
\renewcommand{\(}{\left(}
\renewcommand{\)}{\right)}

\newcommand{\pfrac}[2]{\left(\frac{#1}{#2}\right)}  
\renewcommand{\le}{\leqslant}
\renewcommand{\leq}{\leqslant}
\renewcommand{\ge}{\geqslant}
\renewcommand{\geq}{\geqslant}
\newcommand{\eps}{\ensuremath{\varepsilon}}

\renewcommand{\a}{\alpha}

\newcommand{\g}{\gamma}
\renewcommand{\d}{\delta}
\newcommand{\e}{\varepsilon}
\renewcommand{\l}{\lambda}
\renewcommand{\o}{\omega}

\newcommand{\s}{\sigma}

\begin{document}
	
\title{Primes with small primitive roots}
\author{Kevin Ford, Mikhail R. Gabdullin and Andrew Granville}
\date{}
\address{Department of mathematics, 1409 West Green Street, University of Illinois at Urbana-Champaign, Urbana, IL 61801, USA}
\email{ford126@illinois.edu}
	
\address{Department of mathematics, 1409 West Green Street, University of Illinois at Urbana-Champaign, Urbana, IL 61801, USA; Steklov Mathematical Institute,
Gubkina str., 8, Moscow, 119991, Russia}
\email{mikhailg@illinois.edu, gabdullin.mikhail@yandex.ru}

\address{D\'epartment de Math\'ematiques et Statistique, Universit\'e de Montr\'eal, CP 6128 succ Centre-Ville, Montr\'eal, QC H3C 3J7, Canada.}

\email{andrew.granville@umontreal.ca} 

\begin{abstract}
Let $\d(p)$ tend to zero arbitrarily slowly as $p\to\infty$. 
We exhibit an explicit set $\mathcal{S}$ of primes $p$, defined in terms of
simple functions of the prime factors of $p-1$,
for which the least primitive root of $p$ is $\leq p^{1/4-\d(p)}$ for all $p\in \mathcal{S}$, where  $\#\{p\leq x: p\in \mathcal{S}\} \sim \pi(x)$ as $x\to\infty$. 
\end{abstract}



\thanks{The first author was partially supported by National Science Foundation grants
DMS-1501982, DMS-1802139 and DMS-2301264.
}
\date{\today}
\maketitle

\section{Introduction} 	
In his masterwork \emph{Disquisitiones Arithmeticae}, C. F. Gauss showed that every prime $p$
has a primitive root (that is, a generator of $(\ZZ/p \ZZ)^*)$, and in fact has
$\phi(p-1)$ primitive roots, where $\phi$ is Euler's totient function.  
These have applications to cryptography and pseudo-random number generation \cite[Ch. 8]{CP},
where an efficient algorithm for finding a primitive root is needed. 
In particular, it is of great interest to find upper bounds on 
 $g(p)$, the least primitive root modulo $p$. It is believed that $g(p)=p^{o(1)}$ as $p\to\infty$, and this is considered as a notoriously difficult problem (which is known to be true under the Generalized Riemann Hypothesis, see \cite{An}, \cite{Sh}, \cite{Wang}). The first unconditional upper bounds for $g(p)$ were obtained by  Vinogradov \cite{Vin1}, \cite{Vin3}, Hua \cite{Hua}, and Erd\H{o}s and Shapiro \cite{ES}. The current best estimate is due to 
Wang \cite{Wang},
who established in 1959 that
\begin{equation}\label{0.1} 
g(p) \ll p^{1/4+\e}
\end{equation}
for every fixed $\e>0$.  A key ingredient in the proof is the character sum bound of Burgess
\cite{Bur57}.

On the other hand, if one allows a small exceptional set of primes, much better estimates are known. It was shown by Burgess and Elliott \cite{BE} that
$$
\frac{1}{\pi(x)}\sum_{p\leq x}g(p)\ll (\log x)^2(\log\log x)^4,
$$	
which implies that for any $\eps>0$, most primes $p$ satisfy $g(p)\ll (\log p)^{2+\eps}$.
One can even get such a statement with a tiny exceptional set of primes. Martin \cite{Ma} proved that for any $\e>0$ there is a $C\ll 1/\e$ so that 
$$
g(p) = O\left((\log p)^C\right)
$$ 
with at most $O(x^{\e})$ exceptions $p \leq x$. This is in the spirit of  Linnik's famous theorem \cite{Lin}, which states that the least quadratic non-residue $n(p)$ modulo $p$ satisfies $n(p)\leq p^{\e}$ for all	but $O_{\e}(\log\log x)$ primes $p\leq x$.
 All such statements are ``purely existential'', in  that one cannot say for which specific primes $p$ the bounds hold (say, in terms of the factorization of $p-1$). This motivates the search for explicit conditions on primes $p$ under which (\ref{0.1}) can be refined. 
A first such result appears in \cite{FGK}, where the first author, together with Garaev and Konyagin, proved that for each $r\ge 2$, there is a positive $c_r$ so that
$g(p)<p^{1/4-c_r}$ for all large $p$ with $p-1$ having $r$ distinct prime factors.
Moreover, in \cite{FGK} it is also shown that
$$
g(p)=o(p^{1/4}), \quad p\to\infty
$$
whenever $p-1$ has at most $(0.5-\e)\log\log p$ prime divisors, for any fixed $\e>0$.
As noted in \cite{FGK}, the set of such primes has counting function about 
$x(\log x)^{-3/2+(\log 2)/2-O(\e)}$ and thus (since $-3/2+(\log 2)/2 = -1.153\ldots<-1$) has relative density zero in the set of all primes.  In fact, as was shown by Erd\H os \cite{Erd35},
for most primes $p$, $p-1$ has close to $\log\log p$ prime factors.

 Using different ideas, Sartori \cite{Sar} showed that for every fixed $\e>0$ there exists $y=y(\e)$ such that
$$
g(p)\ll p^{\frac{1}{4\sqrt{e}}+\e}
$$
provided that 
\begin{itemize}
\item all prime divisors of $(p-1)/2$ are greater than $y$;
\item for any divisor $d$ of $p-1$ the bound $J(d)\leq 10\o(d)$ holds (here $J(d)$ is Jacobsthal's function; see section \ref{sec:thm3} for its definition);
\item one has
$$ 
\sum_{\substack{q|p-1 \\ q>2 \mbox{\tiny{ prime}} }}\frac{\log \log q}{\log q}\leq  \e/20.
$$	
\end{itemize}
It is proved in \cite{Sar} that the set of primes enjoying the above three properties has positive relative density (which depends on $\e$).  This bound for $g(p)$ matches, up to an $\eps$ in
the exponent,
the best known upper bound for the least quadratic non-residue modulo $p$, due to 
Burgess \cite{Bur57}; note that a primitive root modulo $p$ is necessarily a quadratic non-residue.
In fact, a number $b$ is a primitive root of $p$ if and only if it is a $q$-th power non-residue
modulo $p$ for all primes $q|(p-1)$.
Comparing this with \cite{FGK}, we see that the bounds in \cite{FGK} are sensitive to the
\emph{number} of prime factors of $p-1$, whereas the methods of \cite{Sar} are sensitive
primarily to the \emph{size} of the odd prime factors of $p-1$.

In this paper we find an explicit set of almost all (that is, of relative density $1$) primes $p$ for which $g(p)=o(p^{1/4})$.  

\begin{thm}\label{thm:main} 
There is an absolute constant $\xi>0$ such that the following holds. For any given large $x$
and for   $\delta\in[(\log x)^{-1/2}, \xi]$, we have 
$$
g(p) \leq p^{1/4-\d}
$$
whenever
 $p\in (\sqrt{x},x]$ is a prime for which the following two conditions hold:
\begin{itemize}
\item[(i)] One has 
$$\sum_{j \le \log(1/\delta)} \frac{1}{q_j} \le \xi \log\log\log (1/\delta);
$$ 
\item[(ii)] There exists $1\le r\leq \frac13 \log(1/\delta)$ such that
$$
\sum_{j>r}\frac{j^3\log\log q_j}{\log q_j}\leq e^{-r-\sqrt{r\log\log\log(1/\delta)}},
$$ 
where $2=q_1<q_2< \cdots <q_{\omega(p-1)}$ are the prime divisors of $p-1$. 
\end{itemize}	
All but $O\Big(\exp \{-(\log\log(1/\delta))^{1/4}\}\pi(x)\Big)$ primes $p\in (\sqrt{x},x]$ satisfy (i) and (ii).
\end{thm}

\smallskip 

We now give a rough explanation why primes failing (i) or (ii) are rare.  By standard results from 
the ``anatomy of integers'', a typical shifted prime $p-1 \le x$ has multiplicative structure similar
to that of a typical integer $n\le x$, specifically there are close to $\log\log t$ prime
factors below $t$, uniformly in $2\le t\le x$; see, e.g., \cite{F24}.
 In particular, for typical primes $p$, $\sum_j 1/q_j$ is bounded.  Moreover, $\log\log q_j$ is about $j$ and thus the sum on the left side of (ii)
is about $e^{-r}$ for most $r$.  However, for reasons stemming from the law of the iterated logarithm
from probability theory, typically there will be some values of $r$ for which the sum is roughly
$\le e^{-r-\sqrt{2r\log\log r}}$.  Thus, if $\delta$ is sufficiently small, we expect (ii) to hold
with high likelihood.

\smallskip

One of the key ingredients of the proof of Theorem \ref{thm:main} is the following refinement of
the main result from \cite{FGK}. 

\begin{thm}\label{thm:beta_r}
The smallest positive integer $n$ which is a simultaneous $p_1,\ldots,p_r$-power nonresidue modulo $p$, where $p_1,\ldots,p_r$ are distinct prime divisors of $p-1$, satisfies
$$
n<p^{1/4-\beta(r,\Phi)}\exp(C(\log p)^{1/2}(\log (2r))^{1/2}),
$$ 
where 
$$
\beta(r,\Phi)=\exp\big(-r-C'\sqrt{r \max(1,\log \Phi)}\big),\quad  \Phi:=\frac{p_1\cdots p_r}{\phi(p_1\cdots p_r)}
$$ 
and $C, C'$ are positive, absolute constants. 	
\end{thm}

From Theorem \ref{thm:beta_r}, we shall deduce the following.

\begin{thm}\label{thm:g(p)-general-upper}
For a given prime $p$ let $2=q_1<q_2< \cdots <q_{\omega(p-1)}$ be the prime divisors of $p-1$.
For any integer $r$ in the range $1\le r\le \omega(p-1)$ we have
\[
g(p) \le p^{1/4-\alpha(p,r)}  
\]
with
\[
\alpha(p,r)=\beta(r,\Phi)
- C'' \bigg( \frac{(\log (2r))^{1/2}}{(\log p)^{1/2}}+
\sum_{j>r} \frac{j^3 \log\log q_j} {\log q_j} \bigg) \text{ and } 
\Phi:=\frac{q_1\cdots q_r}{\phi(q_1\cdots q_r)}
\]
where   $\beta(r,\Phi)$ is defined in Theorem \ref{thm:beta_r} and $C''$ is a positive absolute constant.
\end{thm}

We first prove Theorem \ref{thm:beta_r} in Section \ref{sec:thm2}
by refining the argument from \cite{FGK}.  In Section \ref{sec:thm3}
we deduce Theorem \ref{thm:g(p)-general-upper} from Theorem \ref{thm:beta_r}
using an argument which is optimized for large prime factors of $p-1$.
Section \ref{sec:exceptional} is devoted to estimating the number of primes $\le x$
failing (i) or (ii) in Theorem \ref{thm:main}, and finally Theorem \ref{thm:main}
is proved in Section \ref{sec:main}.

\subsection{Notation.}

With $p\ge 3$ fixed, we denote the distinct prime factors of $p-1$ as $q_1,\ldots,q_{\omega(p-1)}$,
where $q_1=2 < q_2 < \cdots < q_{\omega(p-1)}$.
For prime $q|(p-1)$, let $\cR_q$ be the set of $q$-th power residues modulo $p$, and let $\cN_q$ denote the set of $q$-th power nonresidues modulo $p$,
both considered as subsets of $\{1,2,\ldots,p-1\}$.
We let $P^+(n)$ be the largest prime factor of a positive integer $n$, with $P^+(1)$ defined to be zero, and $P^-(n)$ the smallest prime factor of $n$, with $P^-(1)=\infty$.
We denote $\phi(n)$ for Euler's totient function, $\omega(n)$ for the number of distinct prime factors of $n$, and $\omega(n,a,b)$ for the number of distinct prime factors of $n$ in the interval
$(a,b]$.  We denote $\PR$ for probability and $\E$ for probabilistic expectation.
We also use the notation $\log_k x$ for the $k$-th iterate of
the logarithm and $\exp_k x$ for the $k$-th iterate of the exponential function.
Constants implied by the big-$O$ and $\ll$-symbols are absolute, that is, they are independent
of any parameter unless explicitly stated, e.g. with a subscript to the $O$ or $\ll$
symbol.

%
%
\section{Proof of Theorem \ref{thm:beta_r}}\label{sec:thm2}
%
%

We follow the argument from \cite{FGK} with some technical modifications. First of all, we need to find explicit dependence on $\delta$ in the cornerstone Lemma 4.1 of \cite{FGK}. In \cite{FGK}, the dependence on $\delta$ of the big-$O$ bound is
unspecified, but it can be shown to be $O(1/\delta)$.  Such a bound, however, is
fatal to our argument and we need to refine the proof of Lemma 4.1 in \cite{FGK}
to show a better dependence on $\delta$.

\begin{lem}\label{lem:divisors} 
	Fix $\delta\in(0, \frac1{100}]$ and let $c=\frac{1}{\log 2}+\delta$.
	If $x$ is large enough, and $t$ is a large enough integer with $t\leq (\log x)^{1/c}$
	then all but $O(xt^{-\delta^2/210})$ integers $n\leq x$ have $t$ divisors $d_1<d_2< \cdots <d_t$ for which $d_{j+1}/d_j>x^{1/t^c}$ for all $1\leq j\leq t-1$.  
\end{lem}

\begin{proof} Let $\eps = \delta/2$, so that $c=\frac{1}{\log 2}+2\eps$.
 We may assume that $\e\geq (100\log t)^{-1/2}$, since $t$ is ``large enough'' by the hypothesis.
	Let $\a=\frac{1}{\log 2}+\eps$. 
	We write each $n\leq x$ in the form $n=abd$, where
	$$
	P^-(d)>x^{1/\log t}, \quad P^+(a)\leq x^{1/(t^{\a}\log t)},
	$$	
	and all prime divisors of $b$ lie in 
	$$
	(x^{1/(t^{\a}\log t)}, x^{1/\log t}].
	$$ 
	Let $J=\lceil \a/\e\rceil$, and define
\begin{align*}
\alpha_j &= j\eps \quad (0\le j\le J-1), \;\; \alpha_J = \alpha, \\
A_j &= x^{1/(t^{\a_j}\log t)} \quad (0\le j\le J),\\
\g_j &= \frac{J-j}{3J} \quad (1\le j\le J-1), \;\; \g_0=1.
\end{align*}	
	
	Now we consider the set $S$ of $n\leq x$ which obey the conditions
	\begin{itemize}
		\item[(i')] $d>1$ and $b$ is squarefree; 
		\item[(ii')] $\o(b)\geq k':=\lceil \frac{\log 2t}{\log 2}\rceil$;
		\item[(iii')] for each $j=1,...,J-1$, the number of primes from $(A_j,x^{1/\log t}]$ dividing $n$ is at least $k_j=(1-\g_j)\a_j\log t$. 
	\end{itemize}
We begin by showing that all but a small set of integers $n\leq x$ belong to $S$, by bounding the number of exceptions to (i'), (ii') and (iii'), respectively.
	\smallskip  	
	
(i'):  If $d=1$ then $n$ is $x^{1/\log t}$-smooth.  If $b$ is not squarefree then $n$ is divisible by $p^2$ for some prime
$p>x^{1/(t^\a\log t)}$. Therefore, by Corollary 1.3 in \cite{HiT}, there are at most	
\[
O\pfrac{x}{t}+\sum_{p>x^{1/(t^\a\log t)}}\frac{x}{p^2} \ll \frac{x}{t}+\frac{x}{x^{1/(t^\a\log t)}} 	
\]
integers $n\leq x$ violating (i').  We have $t<x^{1/(t^\a\log t)}$,
since 
$t^{\a}(\log t)^2< t^c\leq \log x$ by the hypothesis,
 and so the number of $n\leq x$ violating (i') is $\ll x/t$.
\medskip
	
(ii'): For each $k<k'$, the number of $n\le x$ with $\omega(b)=k$ is $O(x e^{-E} E^k/k!)$, where
\[
	E:=\sum_{x^{1/(t^{\a}\log t)}<p\leq x^{1/\log t}}\frac1p.
\]
This is a consequence of a  theorem of Hal\'asz \cite{Hal} (see also Theorems 08 and 09 of \cite{HT}).
Mertens theorems imply that $E=\alpha \log t + O(1)$.
Therefore, the number of $n\leq x$ for which (ii') fails is 
	$$
\ll	\sum_{k<k'}xe^{-E}\frac{E^k}{k!}\ll xt^{-\a}\sum_{k<k'}\frac{(\a\log t)^k}{k!}\leq x(t^{-\a})^{Q(\kappa)},
	$$
	where $Q(z)=z\log z-z+1$ and 
	\[
	\kappa=\frac{k'}{\alpha \log t} = \frac{1}{\a\log 2} + O\pfrac{1}{\log t} = \frac{1}{1+\e\log2}+O\pfrac{1}{\log t} \leq 1-\frac23\eps
	\]
	since $t$ is sufficiently large in terms of $\eps$ (recall that $\eps\ge (100\log t)^{-1/2}$).
 Since $Q(1+\l)\geq \frac13 \l^2$ for $|\l|\leq \frac12$, we see that the number of $n\leq x$ violating (ii') is at most
	$$
	\ll xt^{-\frac13 \a(1-\kappa)^2} \ll xt^{-\frac15 \eps^2}.
	$$	
\medskip
	
(iii'): Now fix some $j$ with $1\leq j\leq J-1$ and let $U_{j}$ be the set of $n\leq x$ which have fewer than $k_j$ prime divisors in $(A_j,x^{1/\log t}]$. Arguing similarly, we define
	$$
	E_j:=\sum_{x^{1/(t^{\a_j}\log t)}<p\leq x^{1/\log t}}\frac1p=\a_j \log t+O(1),
	$$	
	and then (since $k_j\leq \a_j\log t$)
	\begin{multline*}
		|U_j|\ll \sum_{k<k_j}xe^{-E_j}\frac{E_j^k}{k!}\ll xt^{-\a_j}\sum_{k<k_j}\frac{(\a_j\log t+O(1))^k}{k!} \\ \ll xt^{-\a_j}\sum_{k<k_j}\frac{(\a_j\log t)^k}{k!} \ll x(t^{-\a_j})^{Q(1-\g_j)} \ll xt^{-\a_j\g_j^2/3}=xt^{-j\e\g_j^2/3}.
	\end{multline*}
	Recalling the definitions of $\g_j$ and $J$, we obtain that the number of $n\leq x$ for which (iii') fails is at most
	\begin{equation}\label{eq:(iii)-violated}
	x\sum_{j=1}^{J-1}t^{-(j\e)((J-j)/J)^2/27}\leq 
	x\sum_{j=1}^{J-1}t^{-(J-j)j^2\e^3/(27 \a^2)};
	\end{equation}
	here we used that $J\geq \a/\e$ and made the substitution $j\mapsto J-j$. 
The summands on the right side of \eqref{eq:(iii)-violated}
 with $j\le J/4$ have total at most
\begin{align*}
\sum_{1\le j\le J/4} t^{-j^2\eps^2/(36\alpha)} \le \frac{t^{-\eps^2/(36\a)}}{1-t^{-\eps^2/(36\alpha)}} \le 2 t^{-\eps^2/(36\a)}.
\end{align*}	
The summand on the right side of \eqref{eq:(iii)-violated}
with $j>J/4$ have total at most
\begin{align*}
\sum_{J/4<j\le J-1} t^{-(J-j)\eps/432} \ll t^{-\eps/432}.
\end{align*}
	By combining the above estimates and using the lower bound for $\e$, we see that all but 
	\begin{equation}\label{1.1}
		\ll x \Big( t^{-\frac15 \eps^2}+ t^{-\e^2/(36\alpha)}+t^{-\eps/432}\Big) \ll xt^{-\e^2/52.5}	= x 	t^{-\delta^2/210}
	\end{equation}
	numbers $n\leq x$ belong to $S$.
	
	Now let $S'$ be the set of $n \in S$ for which $b$ does
	not have $t$ well-spaced divisors in the sense of the lemma. 
	 Since $d>1$ for such $n$, given such a bad value of $b$, using a standard sieve bound for counts of integers without small prime factors (\cite{Ten}, Theorem III.6.4), 	the number of choices for the pair $(a,d)$ is bounded above by
	\begin{align*} 	\sum_{P^+(a)\leq x^{1/(t^{\a}\log t)}}\#\{1<d \leq x/ab :& P^-(d) > x^{1/ \log t} \}\ll \sum_{P^+(a)\leq x^{1/(t^{\a}\log t)}}\frac{x/ab} {\log(x^{1/\log t})} \\	&\ll 	\frac{x\log t}{b\log x}\sum_{P^+(a)\leq x^{1/(t^{\a}\log t)}}\frac1a \ll	\frac{x}{bt^{\a}}. \end{align*}	Hence,
	\begin{equation}\label{1.2}
		|S'| \ll \sum_{\mbox{ \tiny{bad} } b} \frac{x}{bt^{\a}} .
	\end{equation}
	A number $b$ which is bad has many pairs of neighbouring divisors. To be precise, let $\s=t^{-c} \log x$ and define
	$$
	W^*(b; \s) = \#\{(d',d''): d'|b, d''|b, d'\neq d'', |\log (d'/d'')|\leq \s \}.
	$$
	Recall that for bad $b$ we have $\tau(b) \ge 2t$ by (ii').
By the argument leading to (4.2) of \cite{FGK}, 
	\begin{equation}\label{1.3}
		\sum_{\mbox{ \tiny{bad} } b}
		\frac1b \leq 
		\sum_{\mbox{ \tiny{ all } b}}
		\frac{2W^*(b; \s)t}
		{b\tau(b)^2},
	\end{equation}
	each sum being over squarefree integers whose prime factors lie in $(x^{1/(t^{\a} \log t}), x^{1/\log t}]$.
	
	In the sum on the right side of \eqref{1.3}, fix $k = \o(b)$, write $b=p_1 \cdots p_k$, where the $p_i$ are primes and $p_1<...<p_k$. Then $W^*(p_1 \cdots p_k; \s)$ counts the number of pairs $Y, Z \subseteq \{1,..., k\}$ with $Y\neq Z$ and
	\[
		\Big|\sum_{i\in Y}\log p_i - \sum_{i\in Z}\log p_i\Big| \leq \s.
	\]
	Fix $Y$ and $Z$, and let $I$ be the maximum element of the symmetric difference $(Y\cup  Z)\setminus(Y \cap Z)$. We fix $I$ and further partition the solutions according to the condition $A_j<p_I \leq A_{j-1}$, where $1\le j\le J$. Fix the value of $j$.
By the argument leading to \cite[(4.4)]{FGK},	
	\[
		\sum_{x^{1/t^{\a}\log t}<p_1<...<p_k\leq x^{1/\log t}}
		\frac{1}{p_1...p_k}
		\ll\frac{t^{\a_j-c}(\log t)(\a\log t+O(1))^{k-1}}
		{(k-1)!}.
	\]
Since $n\in S$, $I\le k-k_{j-1}$; to cover the $j=1$ case we define
$k_0=0$.  Therefore, by \cite[(4.5)]{FGK},
the number $N(I,j)$ of choices for the pair $Y, Z$ for fixed $I$ and $j$ satisfies
\begin{align*}
	N(I,j) &\leq 4^I 2^{k-I} \le 4^k2^{-k_{j-1}}
	\leq  4^k t^{-(1-\gamma_{j-1})\alpha_{j-1}\log 2}.
\end{align*}
The above bounds, the bound $\alpha_j\le \alpha_{j-1}+\eps$ and the fact that $\log t \ll t^{0.001\e}$ imply
\begin{equation}\label{1.7}
\begin{split}
	\sum_{b}\frac{W^*(b;\s)t}{b\tau(b)^2} &\ll 
		\sum_{j=1}^{J}
		t^{1+\alpha_j - c - (1-\gamma_{j-1})\alpha_{j-1}\, \log 2+0.001\eps} \sum_k
		\frac{(\a\log t+O(1))^{k-1}}
		{(k-1)!} \\
	&\ll t^{1+\a-c+0.7\eps} \sum_{j=1}^J  t^{(1-\log 2 + \gamma_{j-1}\log 2)\alpha_j}.
\end{split}
\end{equation}
Now since  $\a_j=\a_{J-1}-(J-j-1)\e \leq \a-(J-j-1)\e$, we have
\[
\alpha_j(1-\log 2) \le \frac{1}{\log 2} -1 - (J-j-2)(1-\log 2)\eps.
\]
Also, $\alpha_j \g_{j-1} \le \frac13(J-j+1)\eps$, and hence, recalling that
$c=\frac{1}{\log 2}+2\eps$,
\[
(1-\log 2+\g_{j-1}\log2)\a_j \le c-1-2\eps-0.07(J-j)\eps+0.85\eps.
\]
Summing over $j$ and recalling that $t^{\eps} \ge \exp \{\frac{1}{10}\sqrt{\log t}\}$,
we get from (\ref{1.7}) that
	$$
	\sum_b\frac{W^*(b;\s)t}{b\tau(b)^2} \ll t^{\a-0.4\eps}. 
	$$ 
	Thus, by (\ref{1.2}) and (\ref{1.3}), we have
	$$
	|S'| \ll xt^{-0.4\e} = x t^{-0.2 \delta}.
	$$
	Recalling (\ref{1.1}), we see that there are $O(xt^{-\delta^2/210})$ 
	numbers $n \leq x$ for which $b$ does have $t$ well-spaced divisors, and the claim follows.
\end{proof}

Next, we quote the following combinatorial lemma from \cite{FGK}.

\begin{lem}[\cite{FGK}, Corollary 3.4]\label{lem:comb}
Let $p$ be a prime number and suppose $p_1,\ldots, p_r$ are
distinct prime divisors of $p-1$. Let $n$ be a simultaneous
$p_1,\ldots,p_r$-th power nonresidue modulo $p$ and $d_1<\dots<d_t$
be some divisors of $n$ where
$$
t> 2^r\prod_{p_i>2}\frac{p_i}{p_i-1}.
$$
Then there exists $i,j$ such that $1\le i<j\le t$ and the number
$n'=nd_i/d_j$ is also a simultaneous $p_1,\ldots,p_r$-th power
nonresidue modulo $p$.
\end{lem}

The next statement follows from \cite[Lemma 2.2 and (5.1)]{FGK}, and it is a consequence of Burgess' bound for character sums.

\begin{lem}\label{lem:Burgess}
	Let $p$ be a prime number and let $p_1,\ldots,p_r$ be distinct prime factors of $p-1$. Let also $H$ and $m$ be arbitrary positive integers. Then the number of integers $n \leq H$ which are in $\cN_{p_1}\cap \cdots \cap \cN_{p_r}$ is at least
	$$
	\frac{H}{8}\prod_{i=1}^r
	\bigg(1-\frac{1}{p_i}\bigg) - (5r)^{C'''}H^{1-1/m}p^{(m+1)/4m^2}(\log p)^{1/m},
	$$
	where $C'''>0$ is an absolute constant.	
\end{lem}

\bigskip 

Now we are ready to prove Theorem \ref{thm:beta_r}. Let $p$ be large enough. We choose
$$
H=p^{1/4}e^{(C'''+3)(\log p)^{1/2}(\log(5r))^{1/2}} \log p
$$
and
$$
m = \lfloor(\log p)^{1/2}(\log(5r))^{-1/2}\rfloor.
$$
By an identical calculation to that in section 5 of \cite{FGK}, the subtracted term in Lemma \ref{lem:Burgess} is at most $H(5r)^{-2}$.   Let 
\[
 \Phi := \frac{p_1\cdots p_r}{\phi(p_1\cdots p_r)}= \prod_{i=1}^r \(1-\frac{1}{p_i}\)^{-1}.
\]
Crudely,
\begin{equation}\label{eq:Phi}
\Phi^{-1} \ge \prod_{k=2}^{r+1} \bigg(1-\frac{1}{k}\bigg) = \frac{1}{r+1},
\end{equation}
 it follows from Lemma \ref{lem:Burgess} that the set $\cK$ of $n\in [1,H]$ which are in $\cN_{p_1}\cap \cdots \cap \cN_{p_r}$ satisfies
\[
|\cK|\geq 0.04 H \Phi^{-1}.
\]
Let $r<0.6\log_2 p$.
We apply Lemma \ref{lem:divisors} with $x=H$, $c=1/\log2+\delta$ with $\delta\in (0,\frac{1}{100}]$ chosen later, and $t=\max(t_0,\fl{2^r\Phi}+1)$ for some large enough
constant $t_0$.  Since $r\le 0.6\log_2 x$, $t \le \Phi (\log x)^{0.6}+1 \le (\log x)^{1/c}$
if $x$ is large enough. Then, for some absolute $C>0$, all but 
$$
CHt^{-\delta^2/210}
$$ 
integers $n\leq H$ have $t$ well-spaced divisors $d_1< \cdots <d_t$ with 
$d_{j+1}/d_j>x^{t^{-c}}$ for each $j<t$.  We take
$$
\delta= \min \bigg( \frac{1}{100}, \frac{c_3 \max(1,\log \Phi)^{1/2}}{r^{1/2}} \bigg),
$$
with $c_3$ chosen large enough.  We claim that
\begin{equation}\label{eq:end}
CHt^{-\delta^2/210}<0.02H \Phi^{-1}
\end{equation}
if $t_0$ is large enough.  Indeed, if $\delta=\frac{1}{100}$, then by
 \eqref{eq:Phi}, $r$ is bounded, and \eqref{eq:end} holds for large $t_0$.
 If $\delta<\frac{1}{100}$ then 
 \[
 CHt^{-\delta^2/210} \le
 C H \exp \bigg\{ - \frac{c_3^2}{210}(\log 2) \max(1,\log \Phi) \bigg\},
 \]
 and so \eqref{eq:end} holds for large enough $c_3$.

We thus see that there is a simultaneous $p_1,...,p_r$-th power nonresidue $n\leq H$ which has $t$ well-spaced divisors in the sense of Lemma \ref{lem:divisors}. By Lemma \ref{lem:comb}, for some $1\leq i<j\leq t$ the number $nd_i/d_j$ is also a simultaneous $p_1,...,p_r$-th power nonresidue, and
\[
\frac{nd_i}{d_j} \leq H^{1-t^{-c}} \leq Hp^{-t^{-c}/4}.
\]
Since $c=1/\log2+\delta$, we get
\[
t^{-c}/4=\exp\big\{-cr\log2+O(1+\log\Phi)\big\} \geq \exp\big\{ -(r+c_3'(r \max(1,\log \Phi))^{1/2})\big\},
\]
for a sufficiently large constant $c_3'$, upon recalling \eqref{eq:Phi}.
The claim follows. 

It remains to note that for $r \ge 0.6\log_2 p$ the fact that $\cK$ is nonempty 
implies that the least simultaneous $p_1,...,p_r$-th power nonresidue modulo $p$
is $\le H$, and this is good enough since $\beta(r,\Phi) \le (\log p)^{-0.6}$.
 This concludes the proof of Theorem \ref{thm:beta_r}.

%
%
\section{Proof of Theorem \ref{thm:g(p)-general-upper}}\label{sec:thm3}
%
%

The following lemma is based on the work in section 3 of \cite{Sar}.
Jacobsthal’s function $J(m)$ is defined to be the smallest integer $J$ such that any
set of $J$ consecutive integers contains one which is coprime to $m$. 

\begin{lem}\label{lem:nonres-many-q} 
Let $p$ be a prime and let $Q_0, Q_1,..., Q_k$ be disjoint subsets of the prime factors
of $p-1$. For $i=0,1,...,k$, let $m_i=\prod_{q\in Q_i}q$ and suppose that $n_i$ is an integer such that
\begin{itemize}
\item[(a)] $n_i  \in \cN_q$ for all primes $q\in Q_i$;
\item[(b)] $n_i \in \cR_q$ for all primes $q\in Q_0 \cup \cdots \cup Q_{i-1}$.
\end{itemize}
Then there exist integers $a_i$ with $0\leq a_i<J(m_i)$, $i=1,\ldots,k,$ such that
\[
n_0n_1^{a_1}n_2^{a_2}...n_k^{a_k} \in  \cN_q \quad \text{ for all primes } q\in Q_0 \cup \cdots \cup Q_k.
\]
\end{lem}

\begin{proof} We induct on $k\geq 0$. The claim is trivial for $k=0$. 
Now fix $k$ and suppose the claim is true when $k$ is replaced by $k-1$.
Let $Q_0,\ldots,Q_k,n_0,\ldots,n_k$ be as in the lemma, in particular satisfying
(a) and (b).  By the induction hypothesis, there are non-negative integers
$a_1,\ldots,a_{k-1}$ with $a_i < J(m_i)$ for all $i$, and so that
$N:=n_0n_1^{a_1}...n_{k-1}^{a_{k-1}}$ is in $\cN_q$
for all primes $q \in Q_0\cup \cdots \cup Q_{k-1}$.  By (b), $n_k \in \cR_q$ for all primes $q \in Q_0\cup \cdots \cup Q_{k-1}$. Therefore for any $a\geq0$, 
$Nn_k^a \in \cN_q$ for all primes $q \in  Q_0\cup \cdots \cup Q_{k-1}$. 
It remains to find $a$ such that $Nn_k^a\in \cN_q$ for all primes $q \in Q_k$. 
Let $g$ be a primitive root mod $p$ and $N \equiv g^u \pmod p$, $n_k \equiv g^v \pmod p$; then 
\[
Nn_k^a \equiv g^{u+av} \pmod p.
\] 
Since $n_k\in \cN_q$ for all primes $q \in Q_k$, we have $(v,m_k)=1$. 
Let $w$ be an integer for which $wv \equiv 1 \pmod {m_k}$. By the definition of Jacobsthal's function, there exists an integer $a_k$ in the range $0\leq a_k \leq J(m_k)-1$ such that $(uw+a_k, m_k)=1$, and so $(u+a_kv, m_k)=1$ as $u + a_kv \equiv v(uw + a_k) \pmod {m_k}$. But then 
\[
Nn_k^{a_k} \equiv g^{u+a_kv} \pmod p,
\]
and thus $N n_k^{a_k} \in \cN_q$ for all primes $q \in Q_k$, as desired.
\end{proof}

The next incorporates an idea dating back to work of  Vinogradov \cite{Vin2}.

\begin{lem} \label{lem2.2} 
Let $p$ be a large prime and $M\ge 3$ be a divisor of $p-1$, and suppose that every integer less than $Y$ is an $M$-th power residue mod $p$. Then
\[
Y \leq p^{O\big(\frac{\log\log M}{\log M}\big)}.
\]
\end{lem}

\begin{proof} There are exactly $\frac{p-1}{M}$ numbers up to $p-1$ which are $M$-th power residues. On the other hand, the assumption implies that any number $n\leq p-1$ with $P^+(n)\leq Y-1$ is a $M$-th power residue modulo $p$.  By standard counts
of ``smooth numbers'' (e.g., \cite[Theorem 1.2]{HiT}), the number of such $n$ is at least
$(p-1) e^{-O(u\log u)}$, where $u=\frac{\log(p-1)}{\log(Y-1)}$, uniformly in $p\ge 3$
and $1<Y<p$.  Thus, $\log M \ll u\log u$, so $u \gg \frac{\log M}{\log\log M}$ and 
the claim follows.
\end{proof}

\begin{proof}[Proof of Theorem \ref{thm:g(p)-general-upper}]
If $r=\omega(p-1)$ then the claim follows immediately from Theorem \ref{thm:beta_r}.
Now assume that $1\le r\le \omega(p-1)-1$.
For each prime $q | (p-1)$, let $t_q$ be the least element of $\cN_q$,
and write the set $\{t_q : q | (p-1) \}$ as $\{s_0,\ldots,s_\ell\}$, where
$s_0 > s_1 > \cdots > s_\ell$ and let 
\[
\cV_i = \{q | (p-1) : t_q = s_i\} \qquad (0\le i \le \ell).
\] 
Since $s_i$ is the least element
of $\cN_q$ for all $q \in \cV_i$, we see that if $j>i$ then
 $s_j \in \cR_q$ for all $q\in \cV_i$.

Consider an integer $r$  satisfying $1 \le r \le \omega(p-1) - 1$,
and define $h$ as the smallest integer satisfying $r < |\cV_0|+\cdots + |\cV_h|$.
Thus, if $h\ge 1$ then $r \ge |\cV_0|+\cdots +|\cV_{h-1}|$.
We apply Lemma \ref{lem:nonres-many-q}
with $Q_0=\cV_0 \cup \cdots \cup \cV_{h-1}$, $Q_i=\cV_{h+i-1}$ for $1\le i\le \ell-h+1$,
$n_i=s_{i+h-1}$ for $1\le i\le \ell-h+1$, and with $n_0$ the least
integer in $\cN_q$ for all $q\in Q_0$. If $h\ge 1$ the $Q_0$ is nonempty and
\[
n_0 \ge s_0 \ge s_{h-1} > s_{h}=n_1 > \cdots > s_\ell=n_{\ell-h+1},
\]
thus the hypotheses of Lemma \ref{lem:nonres-many-q} hold.  If $h=0$,
 then $Q_0$ is empty and we take $n_0=1$ and the hypotheses of
  Lemma \ref{lem:nonres-many-q} also hold. 
  Consequently, writing $m_i=\prod_{q\in Q_i} q$ for each $i$,
\begin{equation}\label{eq:upper-1}
g(p) \le n_0 \prod_{i=1}^{\ell-h+1} n_i^{J(m_{i})-1}.
\end{equation}
We apply Theorem \ref{thm:beta_r} with $r$ replaced by $r'=|Q_0|$, $\{p_1,\ldots,p_{r'}\}=Q_0$
and $\Phi$ replaced by $\Phi'=q_1\cdots q_{r'}/\phi(q_1\cdots q_{r'})$.
Since $r'\le r$ and $\Phi' \le \Phi=q_1\cdots q_r/\phi(q_1\cdots q_r)$,
\begin{equation}\label{eq:upper-n0-1}
\begin{split}
n_0 &\le p^{1/4-\beta(r',\Phi')} e^{C(\log p)^{1/2} (\log (2r'))^{1/2}} \\ &\le
 p^{1/4-\beta(r,\Phi)} e^{C(\log p)^{1/2} (\log (2r))^{1/2}}.
 \end{split}
\end{equation}
As in Lemma \ref{lem:nonres-many-q}, let
\[
M_i = m_0\cdots m_i \;\; (0\le i\le \ell-h+1).
\]
For each $i\ge 1$, any $n<n_i$ is a $M_i$-th power residue modulo $p$, so by Lemma \ref{lem2.2},
\begin{equation}\label{eq:s_i}
n_i\leq p^{O\left(\frac{\log\log M_i}{\log M_i}\right)}.	
\end{equation}
Here we also used that $\omega(M_i)>r$ so $M_i \ge 6$.
By Iwaniec's theorem \cite{Iw},
\[
J(m) \ll (\o(m) \log (\o(m)+1))^2 \ll (\omega(m))^3
\]
 for all positive integers $m$.
Together with \eqref{eq:upper-1}, \eqref{eq:upper-n0-1} and \eqref{eq:s_i}, we have
\begin{equation}\label{eq:g(p)-gen}
g(p)\leq p^{\theta(p)}, \qquad
\theta(p)=\frac14-\beta(r,\Phi)+O\bigg( \frac{(\log (2r))^{1/2}}{(\log p)^{1/2}}+
\sum_{i=1}^{\ell-h+1} \frac{(\o(m_i))^3 \log\log M_i} {\log M_i} \bigg).
\end{equation}
Recall that $\omega(M_1)>r$.  Thus, for each $i$ we have
\[
\frac{\omega(m_i)^3 \log\log M_i}{\log M_i} \ll \sum_{\max(r,\omega(M_{i-1}))<j \le \omega(M_i)} \frac{j^3 \log \log q_j}{\log q_j}.
\]
Inserting this into \eqref{eq:g(p)-gen} completes the proof of Theorem \ref{thm:g(p)-general-upper}.
\end{proof}

\section{Estimate for the number of exceptional primes}\label{sec:exceptional}

In this section we bound the number of primes below $x$ failing either 
condition (i) or condition (ii) in Theorem \ref{thm:main}.

\subsection{Large values of $1/q_1+\cdots+1/q_r$.}

\begin{lem}\label{lem:sum-recip}
Fix $\xi>0$.  Uniformly for real $R$ with $3\le R \le \log_2 x$, 
the number of primes $p\le x$ such that $\sum_{i\le R} 1/q_i \ge \xi \log_2 R$ is
$O_\xi(\pi(x)/\exp_2 \{ (\log R)^{\xi/6} \})$.
\end{lem}

\begin{proof}
We may assume that $R$ is sufficiently large as a function of $\xi$.
Suppose that $p\le x$ and $1/q_1+\cdots + 1/q_R \ge \xi \log_2 R$.
Let $\cJ$ be the set of integers $j$ with $(\log R)^{\xi/3} < 2^{j} \le  \log R$, and for $j\in \cJ$ let $I_j$ denote the interval $(\exp(2^j),\exp(2^{j+1})]$.
Note that $|\cJ| \le (1/\log 2)\log_2 R+O(1)$.
For some $j\in \cJ$, at least $t_j := \lceil(\xi/3)\exp(2^j)\rceil$ of the $q_i$ lie in $I_j$, for otherwise
\begin{align*}
\sum_{i\le R} \frac{1}{q_i} &\le \sum_{\log q_i\le 2(\log R)^{\xi/3}} \frac{1}{q_i}
+\ssum{i\le R \\ q_i>R} \frac{1}{q_i}
+ \sum_{j\in \cJ} \sum_{i:q_i\in I_j} \frac{1}{q_i} \\
&\le \frac{\xi\log_2 R}{3}+O(1)
+\frac{\xi}{3}|\cJ|<\xi\log_2 R
\end{align*}
if $R$ is large enough.  For $j\in \cJ$, $\exp\{ R\cdot 2^{j+1} \} \le R^{2R} \le 
\exp\{ (\log_2 x)^2 \}$, and also the sum of $\frac{1}{p'-1}$ over all primes $p'\in I_j$
is $\le 1$ for large enough $R$.
Thus by the Brun-Titchmarsh theorem, the number of 
primes $p\le x$ with at least $t_j$ of the $q_i$ in $I_j$ is at most
\begin{align*}
\ll \sum_{k=t_j}^{\fl{R}}\; \ssum{p_1<\cdots<p_k\\ p_i\in I_j\;\forall i} \frac{\pi(x)}{(p_1-1)\cdots (p_k-1)} \le \pi(x) \sum_{k=t_j}^{\fl{R}}\; \frac{1}{k!} \bigg(\sum_{p'\in I_j} \frac{1}{p'-1}\bigg)^k \ll \frac{\pi(x)}{t_j!}.
\end{align*}
As $t_j\ge \exp_2 (j/2)$ for large $j$, the total number of such primes is
$\ll \pi(x) / \exp_3 (j_0/2)$, with $j_0=\min \cJ$.
Since $j_0 \ge (\xi/3)\log_2 R$, the lemma follows.
\end{proof}

\subsection{Applying the law of the iterated logarithm for shifted primes}

\begin{thm}\label{thm:LIL}
Fix $\e\in(0,1]$. Uniformly for 
$1\leq \eta \leq (\frac12+\eps)^{-1}\log_4 x$,
there are at most $O_{\eps}\big(\pi(x) / \exp_2 (0.8\eta\eps) + \pi(x)/\exp_2(\eta/4) \big)$
primes $p\leq x$ for which there is no integer $r$ 
with $1\leq r \leq \exp_2\big((\tfrac12+\e)\eta\big)$ and
$$
\sum_{j>r}\frac{j^3\log_2 q_j}{\log q_j} \leq e^{-r-\sqrt{\eta r}}.
$$
\end{thm} 

The same proof gives a similar statement with the exponent 3 replaced by any fixed constant, but we do not need this here.

We need standard bounds on the tails of the Poisson distribution, e.g.
Norton \cite{Norton}, Lemmas (4.1) and (4.6).

\begin{lem}\label{lem:Poisson-tails}
Let $Y$ be Poisson with parameter $\lambda\ge 1$ and $1\le \alpha \le \lambda^{1/6}$.
Then
\begin{itemize}
\item[(a)] $\PR (Y \le \lambda - \alpha \sqrt{\lambda}) \gg \alpha^{-1} e^{-\frac12 \alpha^2}$;
\item[(b)] $\PR (Y \ge \lambda + \alpha \sqrt{\lambda}) \ll e^{-\frac12 \alpha^2}.$
\end{itemize}
\end{lem}

\begin{proof}[Proof of Theorem \ref{thm:LIL}]
We first relate the quantity in question to a purely probabilistic calculation
using \cite{F24}, and then use ideas from the law of the iterated logarithm
in probability theory to complete the proof.
WLOG suppose $\eta \ge \eta_0(\eps)$, a sufficiently large constant depending
only on $\eps$.  For each integer $j$ let $t_j = \exp (e^j)$, and let
\[
D = \fl{\eta/2}+1, \quad J = \fl{ \log_2 (x^{1/\log_2 x}) } = \fl{\log_2 x-\log_3 x}.
\]
For a randomly selected prime $p\le x$, let $W_j = \omega(p-1,t_j,t_{j+1})$ 
and $W_{a,b} = W_a+\cdots+W_{b-1} = \omega(p-1,t_a,t_b)$, where $j,a,b$ are
non-negative integers.  It is expected by the Kubilius model, and proved in \cite{F24}, that $W_j$
is approximately Poisson with parameter $\lambda_j$,  where
\[
\lambda_j = \sum_{t_j < q \le t_{j+1}} \frac{1}{q-1}.
\]
The prime number theorem with strong error term (cf. Theorem 12.2 in \cite{Ivic}) gives
\begin{equation}\label{eq:lambdaj}
\lambda_j = 1 + O \big( \exp \big\{ - e^{j/2} \big\} \big).
\end{equation}
For each $D\le j\le J$, let $Z_j$ be a Poisson random variable with parameter 
$\lambda_j$, with $Z_D,\ldots,Z_J$ mutually independent, and let $Y_D,\ldots,Y_J$
be mutually independent Poisson random variables each with parameter 1.
Then we expect that $(W_D,\ldots,W_J)$ has distribution close to $(Y_D,\ldots,Y_J)$.
 It is convenient to describe the quality of the approximation using the total
variation distance $d_{TV}(X,Y)$ between two random variables
living on the same discrete space $\Omega$:
\[
d_{TV}(X,Y) := \sup_{A\subset \Omega} \big| \PR(X\in A)-\PR(Y\in A) \big|.
\]
Since $t_{J+1} \le x^{e/\log_2 x}$ and $t_D < \log_2 x$, applying Theorem 2 of \cite{F24}
with $\alpha=1/2$ and $u=(1/e)\log_2 x$, together with \eqref{eq:lambdaj}, gives
\[
d_{TV} \big( (W_D,\ldots,W_J),(Z_D,\ldots,Z_J) \big) \ll \frac{1}{t_D} + \frac{1}{\log x} +u^{-\alpha u}\ll \exp \{ - e^{\eta/2} \}.
\]
Using \eqref{eq:lambdaj}, an easy calculation, e.g. \cite[Lemmas 2.2 and 2.3]{F24}, gives
\[
d_{TV} \big( (Y_D,\ldots,Y_J),(Z_D,\ldots,Z_J) \big) \ll  \sum_{j\ge D} \exp\{ -e^{j/2} \}
\ll \exp \{ -e^{\eta/4} \}.
\]
Therefore, by the triangle inequality,
\begin{equation}\label{eq:WY}
d_{TV} \big( (W_D,\ldots,W_J),(Y_D,\ldots,Y_J) \big) \ll 
 \exp \{ - e^{\eta/4} \}.
\end{equation}

For integers $a,b$ with $D\le a<b\le J$ let $Y_{a,b}=Y_a+\cdots+Y_{b-1}$, so that $Y_{a,b}$
is a Poisson variable with parameter $b-a$, and for disjoint intervals
$[a_1,b_1-1],\ldots,[a_r,b_r-1]$ the variables $Y_{a_1,b_1},\ldots,Y_{a_r,b_r}$
are mutually independent.  Let
\[
K_1 = \fl{e^{\eta/2}}, \qquad K_2 = \lfloor (2\eta)^{-1} e^{(\frac12+\eps)\eta} \rfloor.
\]
In particular, since $\eps \le 1$ and $(\frac12+\eps)\eta \le \log_4 x$, we have
\begin{align*}
K_2! \le  K_2^{K_2} \le \exp \big\{ (3/4) \cdot e^{(\frac12+\eps)\eta} \big\}
\le (\log_2 x)^{3/4} \le J.
\end{align*}
For $K_1 \le k\le K_2-1$ consider the events
\begin{align*}
E_k &: \; W_{k!,(k+1)!} \le k\cdot k! - (1+\eps/10)\sqrt{\eta k\cdot k!}, \\
\cE_k &: \;  Y_{k!,(k+1)!} \le k\cdot k! - (1+\eps/10)\sqrt{\eta k\cdot k!}, \\
F_k &: \ W_{D,k!} \le k!-D+e^{\eta/6} \sqrt{\eta (k!-D)}, \\
\cF_k &:\; Y_{D,k!} \le k!-D+e^{\eta/6} \sqrt{\eta (k!-D)}.
\end{align*}
By \eqref{eq:WY}, the events $E_k$ and $F_k$
corresponding to random primes $p\le x$ are close, respectively, to 
$\cE_k$ and $\cF_k$.  

Consider four more events:
\begin{align*}
A &: \; \text{for some } k\in [K_1,K_2-1], E_k \land F_k, \\
\cA &: \;  \text{for some } k\in [K_1,K_2-1], \cE_k \land \cF_k, \\
B &: \; W_j \le j\text{ for all } j\ge K_1!, \\
\cB &: \; Y_j \le j \text{ for all } j\ge K_1!.
\end{align*}
By \eqref{eq:WY},
\begin{equation}\label{eq:AB-AB}
\PR (A\land B) = \PR( \cA \land \cB) + O\big(\exp \{-e^{\eta/4} \} \big).
\end{equation}
We will show later that $\cA \land \cB$ holds with high probability, and hence so does
$A\land B$.  

Assume for now that $A$ and $B$ both hold.  Then, for some $k$
in the range $K_1 \le k\le K_2-1$, we have $E_k$ and $F_k$.  Then
\begin{align*}
W_{D,(k+1)!} &\le k\cdot k! - (1+\eps/10)\sqrt{\eta k\cdot k!} + k! 
  + e^{\eta/6}\sqrt{\eta k!} \\
&= (k+1)! - \sqrt{\eta (k+1)!} \bigg( (1+\eps/10) \sqrt{\frac{k}{k+1}} - e^{\eta/6}
  (k+1)^{-1/2} \bigg)\\
&\le (k+1)! - (1+\eps/20) \sqrt{\eta (k+1)!}
\end{align*}
as long as $\eta_0(\eps)$ is large enough, and recalling that $k+1 \ge e^{\eta/2}$.
Since 
\[
\omega(p-1,1,t_D) \le t_D \le \exp_2 (\eta/2+1) < (\eps/100)\sqrt{\eta(k+1)!},
\]
we have
\begin{equation}\label{eq:r-upper}
r := \omega(p-1,t_{(k+1)!}) \le (k+1)! - (1+\eps/30)\sqrt{\eta (k+1)!}.
\end{equation}
In particular,
\begin{equation}\label{eq:r2}
r \le K_2! \le \exp \big\{ (3/4) e^{(\frac12 +\eps)\eta} \big\}
\le \exp_2 \{ (\tfrac12+\eps)\eta \}.
\end{equation}
 Then
\begin{align*}
\ssum{i>r \\ q_i \le t_{J+1}} \frac{i^3 \log_2 q_i}{\log q_i} &= \sum_{j= (k+1)!}^J \;\; \sum_{t_j<q_i\le t_{j+1}}  \frac{i^3 \log_2 q_i}{\log q_i}.
\end{align*}
For $j\ge (k+1)!$ and $t_j < q_i \le t_{j+1}$, event $B$ implies that
\[
i \le r + \sum_{\ell=(k+1)!}^j \ell \le \sum_{\ell=1}^j \ell \le j^2.
\]
Since $\frac{\log_2 z}{\log z}$ is decreasing for $z\ge e^e=t_1$, $\frac{\log_2 q_i}{\log q_i} \le
j e^{-j}$.  Since $W_j$ holds, \eqref{eq:r-upper} implies that
\begin{align*}
\ssum{i>r \\ q_i \le t_{J+1}} \frac{i^3 \log_2 q_i}{\log q_i} &\le \sum_{j= (k+1)!}^J 
\frac{j^8}{e^j}
\le 2 ((k+1)!)^8 e^{-(k+1)!} \\
&\le  e^{-r-(1+\eps/30)\sqrt{\eta (k+1)!} + O(k\log k)} \\
&\le e^{-r-(1+\eps/40)\sqrt{\eta (k+1)!}}.
\end{align*}
Since $\omega(p-1,t_{J+1}) \le J^2$ and 
\[
\omega(p-1,t_{J+1},x) \le \frac{\log x}{\log t_{J+1}} \le \log_2 x,
\]
we have $\omega(p-1) \ll (\log_2 x)^2$.
By \eqref{eq:r2}, $r\le K_2! \le (\log_2 x)^{3/4}$ and hence
\[
\sum_{q_i> t_{J+1}}  \frac{i^3 \log_2 q_i}{\log q_i} \ll \big(\omega(p-1)\big)^3\omega(p-1,t_{J+1},x) Je^{-J}
\ll \frac{(\log_2 x)^9}{\log x} \ll e^{-r-2\sqrt{\eta(k+1)!}}.
\]
It follows that
\[
\sum_{i>r} \frac{i^3 \log_2 q_i}{\log q_i} \le e^{-r-\sqrt{\eta (k+1)!}} \le
 e^{-r-\sqrt{\eta r}},
\]
as desired.

It remains to bound the probability that $\cA \land \cB$ holds.
By Lemma \ref{lem:Poisson-tails} (b),
$\PR(\overline{\cF_k}) \ll e^{-\frac12 e^{\eta/3}}$, and thus
the probability that $\cF_k$ fails for some $k\in[K_1,K_2-1]$ is
$\le \exp \{ - e^{0.3 \eta} \}$ for large enough $\eta_0(\eps)$.
By Lemma \ref{lem:Poisson-tails} (a), 
\[
\PR(\cE_k) \gg \eta^{-1/2} e^{-\frac12 (1+\eps/10)^2\eta},
\]
and so if $\eta_0(\eps)$ is large enough then
$\PR(\cE_k) \ge e^{-\frac12 (1+\eps/4)\eta}$.  By independence, the probability that
none of the events $\cE_k$ hold is at most
\begin{align*}
\big(1 - e^{-\frac12(1+\eps/4)\eta}\big)^{K_2-K_1} &\le 
\exp \Big\{ - e^{-\frac12(1+\eps/4)\eta}(K_2-K_1)  \Big\} \\
&\le \exp \big\{ - e^{0.8 \eta \eps} \big\}.
\end{align*}
Thus,
\[
\PR(\cA) \ge 1 - \exp\{ - e^{0.3\eta} \} - \exp \big\{ - e^{0.8 \eta \eps} \big\}.
\]
Since $\PR(Y_j>j)=\sum_{k>j} 1/k! \le 1/j!$, 
\[
\PR(\overline{\cB}) \le \sum_{j\ge K_1!}\; \frac{1}{j!} \le \frac{2}{(K_1!)!} \le
\frac{1}{\exp_2 K_1} \le \frac{1}{\exp_3 (\eta/3)}.
\]
Therefore, 
\[
\PR(\cA \land \cB) \ge 1 - \frac{1}{\exp_2 (0.8\eta \eps)} - \frac{1}{\exp_2(\eta/4)}.
\]
Combining this last estimate with \eqref{eq:AB-AB} completes the proof.
\end{proof}

%
\section{Proof of Theorem \ref{thm:main}}\label{sec:main}
%

 We may assume that $\xi>0$ is small enough, so that $1/\delta$ and $x$ are large enough. 
  Assume that estimates (i)
 and (ii) in Theorem \ref{thm:main} hold with some $r\le \frac13 \log(1/\delta)$.
 Let $\eta=\log_3(1/\delta)$.
 To prove the first claim $g(p) \le p^{1/4-\delta}$, by Theorem \ref{thm:g(p)-general-upper} it suffices to show that $\alpha(p,r) \ge \delta$.
 We first note that  $\log \Phi \le 2(1/q_1+\cdots+1/q_r)\le 2\xi \eta$. 
 If $\xi$ is chosen so that 
 \[
 C' \sqrt{2\xi} \le \tfrac12, \qquad C'' \le \tfrac12 e^{\frac12\sqrt{\log_3(1/\xi)}},
 \]
  then $C'' \le \frac12 e^{\frac12\sqrt{r\eta}}$, and
   by (i) and (ii),
 \begin{align*}
 \alpha(p,r) &\ge e^{-r-\frac12 \sqrt{r\eta}} - C''
 \bigg( \pfrac{\log (2r)}{\log p}^{1/2} + e^{-r-\sqrt{r\eta}} \bigg)\\
 &\ge \tfrac12\, e^{-r-\frac12 \sqrt{r\eta}} - C''\pfrac{\log (2r)}{\log p}^{1/2}.
 \end{align*}
By assumption, $r\le \frac13 \log(1/\delta) \le \frac16 \log_2 x$, and thus (recall that $p>\sqrt{x}$)
\[
C'' \pfrac{\log (2r)}{\log p}^{1/2} \le C'' \pfrac{\log_2 x}{\log x}^{1/2} \le \frac{1}{(\log x)^{1/3}}.
\]
Since $x$ is sufficiently large and $\d\geq (\log x)^{-1/2}$ is sufficiently small, we conclude that 
$$
\alpha(p,r) \ge 0.5e^{-0.4\log(1/\d)}-(\log x)^{-1/3} \ge \delta.
$$ 

We now bound the number of primes $p\le x$ failing (i) or (ii) in Theorem \ref{thm:main}.  Let $R=\log(1/\delta)$, so that by assumption $R\le \frac12\log_2 x$.
By Lemma \ref{lem:sum-recip}, the number of primes $p\le x$ failing (i) is
 \[
 \ll \frac{\pi(x)}{\exp_2 \{ (\log_2(1/\delta))^{\xi/6} \}},
 \]
 which is much smaller than the bound given in Theorem \ref{thm:main}.
 
 Take $\eps=0.49$.
We have $\eta \le \log_4 x$ and $\frac13 \log(1/\delta) \ge \exp_2((\frac12+\eps)\eta)$
for small enough $\delta$.  Theorem \ref{thm:LIL} then implies that
the number of primes $p\le x$ failing (ii) is
\[
\ll \frac{\pi(x)}{\exp_2 (\eta/4)} = \frac{\pi(x)}{\exp \{ (\log_2(1/\delta))^{1/4} \}}.
\]
This completes the proof of Theorem \ref{thm:main}.

\end{document}